\documentclass[12pt,oneside,openany,article]{memoir}

\usepackage{mempatch}

\nouppercaseheads
\usepackage{amsthm}
\usepackage[noTeX]{mmap}
\usepackage{etex}

\usepackage[english]{babel}
\usepackage[leqno]{mathtools}

\usepackage{mathrsfs}
\usepackage{upgreek}
\usepackage[compress,square,comma,numbers]{natbib}
\usepackage{hypernat} 

\usepackage{sfmath}
\usepackage{tikz}

\usepackage{microtype}

\usepackage{array}

\usepackage{graphicx,xcolor}
\definecolor{refkey}{rgb}{0,0,1}
\definecolor{labelkey}{rgb}{1,0,0}

\usepackage{hyperxmp}
\usepackage{xr-hyper}
\usepackage{nameref}
\usepackage[pdftex,bookmarks,pdfnewwindow,plainpages=false,unicode]{hyperref}

\usepackage{bookmark}

\usepackage{enumitem}
\usepackage{amssymb}

\hypersetup{
colorlinks=true,
linkcolor=black,
citecolor=black,
urlcolor=blue,
pdfauthor={Andrew Hassell, Victor Ivrii},
pdftitle={Semiclassical Dirichlet to Neumann operator},
pdfsubject={Sharp Spectral Asymptotics},
pdfkeywords={Microlocal Analysis,  Sharp Spectral Asymptotics, Periodic flows},
bookmarksdepth={4}
}

\theoremstyle{plain}
\newtheorem{theorem}{Theorem}[chapter]
\newtheorem{lemma}[theorem]{Lemma}
\newtheorem{proposition}[theorem]{Proposition}

\theoremstyle{remark}
\newtheorem{remark}[theorem]{Remark}

\numberwithin{equation}{chapter}

\newenvironment{phantomequation}[1][]{\refstepcounter{equation}}{}

\newcounter{note}

\newcommand{\scl}{\mathsf{scl}}
\newcommand{\bR}{\mathbb{R}}

\newcommand{\bZ}{\mathbb{Z}}

\newcommand{\sub}{\operatorname{sub}}
\newcommand{\vol}{\operatorname{vol}}
\newcommand\arccot{\operatorname{arccot}}

\newcommand{\N}{{\mathsf{N}}}
\newcommand\RR{\mathbb{R}}
\newcommand\CC{\mathbb{C}}
\newcommand{\Def}{\mathrel{\mathop:}=}
\renewcommand{\Im}{\operatorname{Im}}
\renewcommand{\Re}{\operatorname{Re}}

\begin{document}

\addtopsmarks{headings}{}{%
\createmark{chapter}{right}{shownumber}{}{. \ }
}
\pagestyle{headings}

\title{Spectral asymptotics for the semiclassical Dirichlet to Neumann operator
}

\author{Andrew Hassell\thanks{This research was supported in part by Australian Research Council Discovery Grant  DP120102019}
,  Victor Ivrii\thanks{This research was supported in part by National Science and Engineering  Research Council (Canada) Discovery Grant  RGPIN 13827}}



\maketitle

\begin{abstract}
Let $M$ be a compact Riemannian manifold with smooth boundary, and let $R(\lambda)$ be the Dirichlet-to-Neumann operator at frequency $\lambda$. The semiclassical Dirichlet-to-Neumann operator $R_{\scl}(\lambda)$ is defined to be $\lambda^{-1} R(\lambda)$. We obtain a leading asymptotic for the spectral counting function for $R_{\scl}(\lambda)$ in an interval $[a_1, a_2)$ as $\lambda \to \infty$, under the assumption that the measure of periodic billiards on $T^*M$ is zero.
The asymptotic takes the form 
\begin{equation*}
\N(\lambda; a_1,a_2) = \bigl( \kappa(a_2)-\kappa(a_1)\bigr)\vol'(\partial M) \lambda^{d-1}+o(\lambda^{d-1}),
\label{eq-1-19}
\end{equation*}
where $\kappa(a)$ is given explicitly by
\begin{multline*}
\kappa(a) = \frac{\omega_{d-1}}{(2\pi)^{d-1}} \bigg( -\frac{1}{2\pi} \int_{-1}^1 (1 - \eta^2)^{(d-1)/2}  \frac{a}{a^2 + \eta^2}  \,  d\eta  \\
- \frac{1}{4} + H(a)   (1+a^2)^{(d-1)/2} \bigg) .
\label{eq-1-20}
\end{multline*}
\end{abstract}

\chapter{Introduction}
\label{sect-1}

Let $M$ be a Riemannian manifold with boundary. The Dirichlet-to-Neumann operator is a family of operators defined on $L^2({\partial M})$ depending on the parameter $\lambda \geq 0$. It is defined as follows: given $f \in L^2({\partial M})$, we solve the equation (if possible)
\begin{equation}
(\Delta - \lambda^2) u = 0 \quad \text{on\ \ } M, \qquad u |_{{\partial M}} = f.
\label{eq-1-1}
\end{equation}
Then the Dirichlet-to-Neumann operator at frequency $\lambda$ is
the map
\begin{equation}
R (\lambda) : f \mapsto -\partial_\nu u |_{{\partial M}}.
\label{eq-1-2}
\end{equation}
Here $\partial_\nu$ is the interior unit normal derivative, and $\Delta$ is the positive Laplacian on $M$.

It is well known that $R(\lambda)$ is a self-adjoint, semi-bounded from below pseudodifferential operator of order $1$ on $L^2({\partial M})$, with domain $H^1(\partial M)$. It therefore has discrete spectrum accumulating only at $+\infty$.
The Dirichlet-to-Neumann operator and closely related operators are important in a number of areas of mathematical analysis including inverse problems (such as Calder\'on's problem \cite{Cal}), domain decomposition problems (such as the determinant gluing formula of Burghelea-Friedlander-Kappeler \cite{BFK}), and spectral asymptotics (see e.g. \cite{Fried}). 

In this paper, we are interested in the spectral asymptotics of $R(\lambda)$ in the high-frequency limit, $\lambda \to \infty$. Let us recall standard spectral asymptotics for elliptic differential operators, for simplicity in the simplest case of a positive self-adjoint second order scalar operator. Suppose that $Q$ is such an operator on a manifold $M$ of dimension $d$, with principal symbol $q$. Then in the case that $M$ is closed, we have an asymptotic for the number $\N(\lambda)$ of eigenvalues of $Q$ (counted with multiplicity) less than $\lambda^2$:
\begin{multline}
\N(\lambda) = (2\pi)^{-d} \vol \{ (x, \xi) \in T^* M \mid q(x, \xi) \leq \lambda^2 \} + O(\lambda^{d-1}) \\
= \big( \frac{\lambda}{2\pi} \big)^d \vol \{ (x, \xi) \in T^* M \mid q(x, \xi) \leq 1 \} + O(\lambda^{d-1}),
\label{eq-1-3}
\end{multline}
where the equality of the two expressions on the RHS is a simple consequence of the homogeneity of $q$. Moreover, if the set of periodic geodesics has measure zero, then there is a two-term expansion of the form
\begin{equation*}
\bigl( \frac{\lambda}{2\pi} \bigr)^d \vol \{ (x, \xi) \in T^* M \mid q(x, \xi) \leq 1 \} +
\frac{\lambda^{d-1}}{(2\pi)^{d}} \int_{ \{ q = 1 \}} \sub(Q) + o(\lambda^{d-1})
\end{equation*}
where $\sub(Q)$ is the subprincipal symbol of $Q$ \cite{DuGu}. This was generalised to the case of manifolds with boundary by the second author \cite{Ivr80}. For simplicity we state the result in the case that $Q = \Delta$ is the (positive) metric Laplacian, which satisfies $\sub(\Delta) = 0$. Then $\Delta$ is a self-adjoint operator under either Dirichlet $(-)$ or Neumann $(+)$ boundary conditions, and if the set of periodic generalised bicharacteristics has measure zero, we get a two-term expansion for $\N_\Delta(\lambda)$  of the form
\begin{equation}
\bigl( \frac{\lambda}{2\pi} \bigr)^d \vol B^* M \pm
\frac{1}{4} \big(\frac{\lambda}{2\pi} \big)^{d-1} \vol B^* \partial M + o(\lambda^{d-1}).
\label{eq-1-4}\end{equation}

These statements can be generalized to the semiclassical setting. Consider a classical Schr\"odinger operator on $M$, $P = h^2 \Delta + V(x) - 1$, where $h > 0$ is a small parameter (``Planck's constant'') and $V$ is a smooth real-valued function. We consider the asymptotic behaviour $\N^-_h(P)$ of the number of negative eigenvalues of $P$ as $h \to 0$.
This is equivalent to the problem above if $h = \lambda^{-1}$ and $V$ is identically zero. Define $p(x, \xi)$ to be the semiclassical symbol of $P$, i.e. $p = |\xi|^2_{g(x)} + V(x) - 1$. Then, if $M$ is closed, under the assumption that the measure of periodic bicharacteristics of $P$ is zero in $T^* M$, and that $0$ is a regular value for $p$, we have
\begin{equation}
\N^-_h(P) =
(2\pi h)^{-d} \vol \{ (x, \xi) \in T^* M \mid p(x, \xi) \leq 0 \}  + O(h^{1-d}).
\label{eq-1-5}
\end{equation}
Moreover, for manifolds with boundary, we have an analogue of \eqref{eq-1-4}:
under either Dirichlet $(-)$ or Neumann $(+)$ boundary conditions, if the set of periodic generalised bicharacteristics has measure zero, we get a two-term expansion for $\N^-_h(P)$ (where here we understand the self-adjoint realization of $P$ with either Dirichlet or Neumann boundary condition)  of the form
\begin{equation}
(2\pi h)^{-d} \vol \{ (x, \xi) \in T^* M \mid p(x, \xi) \leq 0 \}  \pm \frac{1}{4} (2\pi h)^{1-d} \vol  \mathcal{H}  + o(h^{1-d}),
\label{eq-1-6}\end{equation}
where $\mathcal{H} \subset T^* (\partial M)$ is the hyperbolic region in the boundary, that is, the projection of the set
$\{(x,\xi) \mid p(x, \xi) \leq 0\} \cap T^*_{\partial M} M$ to $T^* \partial M$.

From the semiclassical point of view, since $R(\lambda)$ is a first order operator, it makes sense to consider $R_\scl(\lambda) := \lambda^{-1} R(\lambda)$ (for $\lambda > 0$), which we call the semiclassical Dirichlet-to-Neumann operator. Like $R(\lambda)$, it is a self-adjoint, semi-bounded from below operator on $L^2({\partial M})$, with discrete spectrum accumulating only at $+\infty$. The goal of this paper is to investigate the spectral asymptotics of $R_\scl(\lambda)$, that is, the asymptotics of
\begin{equation}
\N(\lambda; a_1,a_2)\Def  \#\{\beta:\, \beta \text{\  is an eigenvalue of \  } R_\scl (\lambda),\ a_1\le \beta <a_2\},
\label{eq-1-7}
\end{equation}
the number of eigenvalues of $R_\scl(\lambda)$ in the interval $[a_1, a_2)$, as $\lambda \to \infty$.

Both $R(\lambda)$ and $R_\scl(\lambda)$ have the disadvantage that they are undefined whenever $\lambda^2$ is a Dirichlet eigenvalue, since then \eqref{eq-1-1} is not solvable for arbitrary $f \in H^1(M)$. Indeed, when $\lambda^2$ is a Dirichlet eigenvalue, a necessary condition for solvability of (\ref{eq-1-1}) is that $f$ is orthogonal to the normal derivatives of Dirichlet eigenfunctions at frequency $\lambda$. To overcome this issue, we introduce the \emph{Cayley transform} of $R_\scl(\lambda)$: we define
\begin{equation}
C(\lambda) =  (R_\scl (\lambda)- i)(R_\scl (\lambda)+ i)^{-1}.
\label{eq-1-8}
\end{equation}
This family of operators is related to impedance boundary conditions: we have
$C(\lambda) f = g$ if and only if there is a function $u$ on $M$ satisfying 
\begin{gather}
(\Delta - \lambda^2) u = 0
\label{eq-1-9}\\
\shortintertext{and}
\frac{1}{2} (\lambda^{-1} \partial_\nu u- i u) = f,
\qquad \frac{1}{2} (\lambda^{-1} \partial_\nu  u + i u) = g.
\label{eq-1-10-*}\tag*{$\textup{(\ref*{eq-1-10})}_{1,2}$}
\end{gather}
\begin{phantomequation}\label{eq-1-10}\end{phantomequation}
As observed in \cite{BH}, $C(\lambda)$ is a well-defined analytic family of operators for $\lambda$ in a neighbourhood of the positive real axis, which is unitary on the real axis. In particular, it is well-defined even when $\lambda^2$ is a Dirichlet  eigenvalue of the Laplacian on $M$. As a unitary operator, $C(\lambda)$, $\lambda > 0$, has its spectrum on the unit circle, and as $R_\scl(\lambda)$ has discrete spectrum accumulating only at $\infty$, it follows that the spectrum of $C(\lambda)$ is discrete on the unit circle except at the point $1$. Our question can be formulated in terms of $C(\lambda)$: given two angles $\theta_1, \theta_2$ satisfying $0 < \theta_1 < \theta_2 < 2\pi$, what is the leading asymptotic for
\begin{equation}
\tilde\N(\lambda; \theta_1,\theta_2)\Def \# \{ e^{i\theta}:\, e^{i\theta} \text{\  is an eigenvalue of \  } C(\lambda),\ \theta_1\le \theta <\theta_2\}
\label{eq-1-11}
\end{equation}
the number of eigenvalues of $C(\lambda)$ in the interval $\{ e^{i\theta}:\, \theta \in [\theta_1, \theta_2) \}$ of the unit circle, as $\lambda \to \infty$. Clearly, we have
\begin{equation}
\tilde{\N}(\lambda; \theta_1,\theta_2) =  \N(\lambda; a_1,a_2), \
\text{where } \ e^{i\theta_j} = \frac{a_j - i}{a_j + i}, \text{ i.e. }
a_j = - \cot \bigl(\frac{\theta_j}{2}\bigr) .
\label{eq-1-12}
\end{equation}
To answer this question we relate it to a standard semiclassical eigenvalue counting problem on $M$. To state the next result, we first define the self-adjoint operator $P_{a,h}$ on $L^2(M)$ by
\begin{gather}
\mathfrak{D}(P_{a,h})= \{ u \in H^2(M) :\,  (h \partial_\nu   + a ) u = 0 \text{\ \ at\ \ } {\partial M}  \},
\label{eq-1-13}\\
P_{a,h}(u) = (h^2 \Delta - 1) u, \quad u \in \mathfrak{D}(P_{a,h}).
\label{eq-1-14}
\end{gather}
It is the self-adjoint operator associated to the semi-bounded quadratic form
\begin{equation}
h^2 \|\nabla u\|_M ^2 - \|u\|_M ^2   - h a  \|u\|^2_{\partial M} .
\label{eq-1-15}
\end{equation}
The operator $P_{a,h}$ is linked with the semiclassical Dirichlet-to-Neumann operator as follows:  if $f$ is
an eigenfunction of $R_\scl(\lambda)$ with eigenvalue $a$,  then the corresponding Helmholtz function $u$ defined by \eqref{eq-1-1} is in the domain \eqref{eq-1-13} of $P_{a,h}$, and $P_{a,h} u = 0$ (where $h = \lambda^{-1}$).

Then we have  the following result, proved in Section~\ref{sect-3}.

\begin{proposition}\label{prop-1-1}
Let $h = \lambda^{-1}$.  Assume $0 < \theta_1 < \theta_2 < 2\pi$.  Then the number of eigenvalues of $C(\lambda)$ in the interval
$J_{\theta_1, \theta_2} := \{ e^{i\theta}:\, \theta \in [\theta_1, \theta_2) \}$  is equal to
\begin{gather}
\tilde{\N}(\lambda; \theta_1,\theta_2) = \N(\lambda; a_1,a_2)= \N_h^- (a_2) - \N^-_h(a_1) ,
\label{eq-1-16}\\
\intertext{where $a_j = - \cot (\theta_j/2)$ and}
\N_h^- (a)\Def \# \{\mu:\, \mu \text{\ is an eigenvalue of\ \ } P_{a,h} ,\ \mu <0\}
\label{eq-1-17}.
\end{gather}
\end{proposition}

Having thus reduced the problem to a standard question about semiclassical spectral asymptotics, we obtain (after some calculations in Section~\ref{sect-4}) our main result.

\begin{theorem}\label{thm-1-2}
\begin{enumerate}[fullwidth, label=(\roman*)]
\item\label{thm-1-2-i}
The following estimate for the quantity \eqref{eq-1-11} holds:
\begin{equation}
\N(\lambda; a_1,a_2) = O(\lambda^{d-1});
\label{eq-1-18}
\end{equation}
\item\label{thm-1-2-ii}
Further, if the set of periodic billiards on $M$ has measure $0$ then the following asymptotic holds as  $\lambda \to +\infty$:
\begin{equation}
\N(\lambda; a_1,a_2) = \bigl( \kappa(a_2)-\kappa(a_1)\bigr)\vol'(\partial M) \lambda^{d-1}+o(\lambda^{d-1}),
\label{eq-1-19}
\end{equation}
where $\kappa(a)$ is given explicitly by
\begin{multline}
\kappa(a) = \frac{\omega_{d-1}}{(2\pi)^{d-1}} \bigg( -\frac{1}{2\pi} \int_{-1}^1 (1 - \eta^2)^{(d-1)/2}  \frac{a}{a^2 + \eta^2}  \,  d\eta  \\
- \frac{1}{4} + H(a)   (1+a^2)^{(d-1)/2} \bigg) .
\label{eq-1-20}
\end{multline}
Here  $H(\cdot)$ is the Heaviside function,
$\omega_d$ is the volume of the unit ball in $\bR^d$, and $\vol (M)$ and $\vol'(\partial M)$ are $d$-dimensional volume of $M$ and $(d-1)$-dimensional volume of $\partial M$  respectively.

\item\label{thm-1-2-iii} In the case $d=3$, we can evaluate this integral exactly and we find that
\begin{equation}
\kappa (a)=\frac{1}{4\pi}\Bigl(-\frac{1}{4}-\frac{1}{\pi}\arccot (a) (1+a^2)+(1+a^2)+ \frac{1}{\pi}a\Bigr)
\label{eq-1-21}
\end{equation}
where $\arccot$ has range $(0,\pi)$.
This is simpler expressed in terms of $\theta$. Defining $\tilde\kappa(\theta) = \kappa(a)$ where $a = -\cot(\theta/2) = \cot(\pi - \theta/2)$, we have (still under the zero-measure assumption on periodic billiards)
\begin{gather}
\tilde{\N}(\lambda; \theta_1,\theta_2) = \bigl( \tilde{\kappa}(\theta_2)-\tilde\kappa(\theta_1)\bigr)\vol'(\partial M) \lambda^{2}+o(\lambda^{2}),
\label{eq-1-22}\\[2pt]
\tilde{\kappa}(\theta) = \frac{1}{4\pi} \Big( -\frac{1}{4} + \frac{1}{2\pi} \big( \frac{\theta - \sin \theta}{\sin^2(\theta/2)} \big) \Big).
\label{eq-1-23}\end{gather}
\end{enumerate}
\end{theorem}

\begin{remark}\label{rem-1-3}
\begin{enumerate}[fullwidth, label=(\roman*)]
\item\label{rem-1-3-i}
It looks disheartening that the remainder is just ``$o$'' in comparison with the principal part and only under geometric condition of the global nature but it is the nature of the beast. Consider $M$ a hemisphere; then for $\lambda_n^2= n(n+d-1)$ with $n\in \bZ^+$ the operator $R(\lambda)$ has eigenvalue $0$ of multiplicity $\asymp n^{d-1}$ and therefore we do not have even an asymptotic but just an estimate.
\item\label{rem-1-3-ii}
We can explain this by pointing out that the problem we consider is truly $d$-dimensional, in contrast, for example, to the problem of distribution of negative eigenvalues of $H=\Delta$ with  \begin{equation*}
\frak{D}(H)=\{u\in H^2 (X):\,\partial_\nu u|_{\partial M}+ a \Delta_b^{1/2} u|_{\partial M}=0\}
\end{equation*}
with $\Delta_b$ positive Laplacian on $\partial M$ and $a> 1$ which is in fact $(d-1)$-dimensional.

\item\label{rem-1-3-iii}
On the other hand, calculations of Section~\ref{sect-5} imply that for $a\ne 0$ quantum periodicity will be broken (see \cite{futurebook}, Section \ref{book_new-sect-8-3}) and one can get remainder estimate $o(\lambda^{d-1})$ albeit with the oscillating principal part.

\item\label{rem-1-3-iv}
Under certain assumptions about geodesic billiard flow remainder (see \cite{futurebook}, Section \ref{book_new-sect-7-4}) estimate (\ref{eq-4-2}) could be improved to $O(h^{1-d}/|\log h|)$ or even $O(h^{1-d+\delta})$ with a small exponent $\delta>0$ and therefore remainder estimate (\ref{eq-1-19}) could be improved to $O(\lambda^{d-1}/\log \lambda )$ or even $O(\lambda^{d-1-\delta})$.

\item\label{rem-1-3-v}
Despite different expressions for $a<0$ and $a>0$ one can observe that $\kappa (a)$ is monotone increasing and continuous.

\item\label{rem-1-3-vi}
In particular, $\kappa (0)=\frac{1}{4} (2\pi)^{-d}\omega_{d-1}$ and $\kappa (-\infty)=-\frac{1}{4} (2\pi)^{-d}\omega_{d-1}$. Since $a=0$ corresponds to the Neumann boundary condition and $a = -\infty$ to Dirichlet (see Proposition~\ref{prop-5-1}), this agrees with \eqref{eq-1-4}.

\item\label{rem-1-3-vii}
One can consider eigenvalues of operator $\rho \lambda^{-1}R(\lambda)$ with $\rho>0$ smooth on $\partial M$;  then estimates (\ref{eq-4-1}), (\ref{eq-1-18})  and asymptotics (\ref{eq-4-2}), (\ref{eq-1-19}) hold in the frameworks of Statement \ref{thm-1-2-i} and   \ref{thm-1-2-ii} of Theorem~\ref{thm-1-2} respectively albeit with $\kappa(a) \vol'(\partial M)$ replaced by
\begin{equation*}
\int_{\partial M} \kappa (\rho (x')a)\,d\sigma
\end{equation*}
where $d\sigma$ is a natural measure on $\partial M$;
however without this condition $\rho>0$ problem may be  much more challenging; even self-adjointness is by no means guaranteed.

\item\label{rem-1-3-viii}
Operators of the form $P_{a, h}$ were considered by Frank and Geisinger \cite{FG}. They showed that the trace of the negative part of $P_{a,h}$ has a two-term expansion as $h \to 0$ regardless of dynamical assumptions\footnote{The fact that Frank and Geisinger obtain a second term regardless of dynamical assumptions is simply due to the fact that they study $\operatorname{Tr} f(P_{a,h})$ with $f(\lambda) = -\lambda H(-\lambda)$ ($H$ is the Heaviside function), which is one order smoother than $f(\lambda) = H(-\lambda)$.}, and the second term in their expansion (the $L_d^{(2)}$ term of \cite[Theorem 1.1]{FG}) is closely related to $\kappa(a)$ --- see Remark~\ref{rem:FG}. 
\end{enumerate}
\end{remark}

\begin{remark}\label{rem-1-4}
We can rephrase Theorem~\ref{thm-1-2} in terms of a limiting measure on the unit circle. For each $\lambda > 0$, let $\mu(h)$, $h = \lambda^{-1}$, denote the atomic measure determined by the spectrum of $C(\lambda)$:
\begin{equation}
\mu(h) = (2\pi h)^{d-1} \sum_{e^{i\theta_j} \in \mathrm{spec} C(h^{-1})} \delta(\theta - \theta_j),
\label{eq-1-24}\end{equation}
where we include each eigenvalue according to its multiplicity as usual. Then Theorem~\ref{thm-1-2} can be expressed in the following way: the measures $\mu(h)$ converge in the weak-$*$ topology as $h \to 0$ to the measure
\begin{equation}
\omega_{d-1} \vol'(\partial M) \frac{d}{d\theta} \tilde \kappa(\theta) d\theta \quad \text{on } (0, 2\pi), \text{ that is on } S^1 \setminus \{ 1 \}.
\end{equation}
In particular, this measure is absolutely continuous, and finite away from $e^{i\theta} = 1$ with an infinite accumulation of mass as $\theta \uparrow 2\pi$. In this form, we can compare our result with results on the semiclassical spectral asymptotics of scattering matrices. In \cite{DGHH} and \cite{GHZ}, the scattering matrix $S_h(E)$ at energy $E$ for the Schr\"odinger operator  $h^2 \Delta + V(x)$ on $\RR^d$ was studied in the semiclassical limit $h \to 0$. Assuming that $V$ is smooth and compactly supported, that $E$  is a nontrapping energy level, and that the set of periodic trajectories of the classical scattering transformation on $T^* S^{d-1}$ has measure zero, it was shown that the measure $\mu(h)$ defined by \eqref{eq-1-24} converged weak-$*$ to a uniform measure on $S^1 \setminus \{1 \}$, with an atom of infinite mass at the point $1$. On the other hand, for polynomially decaying potentials, it was shown by Sobolev and Yafaev \cite{SoYa} in the case of central potentials and by Gell-Redman and the first author more generally (work in progress) that there is a limiting measure which is nonuniform, and is qualitatively similar to the measure for $C(h^{-1})$ above in that it is finite away from $1$, with  an infinite accumulation of mass at $1$ from one side.
\end{remark}


\chapter{Reduction to semiclassical spectral asymptotics}
\label{sect-3}

In this section we prove Proposition~\ref{prop-1-1}. This result actually follows directly   from the Birman-Schwinger principle. As some readers may not be familiar with this, we give the details. 

\begin{proof}[Proof of Proposition~\ref{prop-1-1}] 
We begin by recalling that the operator $P_{a, h}$ is the self-adjoint operator associated to the quadratic form \eqref{eq-1-15}, that is, 
\begin{equation*}
Q_{a,h}(u) \Def h^2 \|\nabla u\|_M ^2 - \|u\|_M ^2   - h a  \|u\|^2_{\partial M} .
\end{equation*}
We recall the min-max characterization of eigenvalues: the $n$th eigenvalue $\mu_n(a, h)$ of $P_{a,h}$ is equal to the infimum of 
$$
\sup_{ v  \in V, \| v \| = 1 } Q_{a,h}(v)
$$
over all subspaces $V \in H^1(M)$ of dimension $n$. 
The monotonicity of $Q_{a,h}$ in $a$, for fixed $h$, shows that the eigenvalues are monotone nonincreasing with $a$. In fact, they are strictly decreasing, which follows from the fact that eigenfunctions of $P_{a, h}$ cannot vanish at the boundary. Indeed, the eigenfunctions satisfy the boundary condition $h \partial_\nu  u = -a u$, which shows that if $u$ vanishes at the boundary, so does $\partial_\nu u$, which is impossible. 

The eigenvalues $\mu_n(a, h)$ are thus continuous, strictly decreasing functions of $a$. Let $a_1 < a_2$ be real numbers. The Birman-Schwinger principle \cite[Prop. 9.2.7]{Ivr1} says that the number of negative eigenvalues of $P_{a_2, h}$ is equal to the number of negative eigenvalues of $P_{a_1, h}$, plus the number of eigenvalues $\mu_n(a,h)$ of $P_{a, h}$ that change from nonnegative to negative as $a$ varies from $a_1$ to $a_2$. A diagram makes this clear: see Figure~\ref{evals}. 

\begin{figure}
\centering
\begin{tikzpicture}

\draw[very thick, ->] (-1,0)--(8,0) node[below]{$a$};
\draw[thin,->] (-.7,-4)--(-.7,4) node[right]{$\mu$};
\draw[very thick] (1,-4)--(1,4) node[right]  {$a=a_1$};
\draw[very thick] (6,-4)--(6,4) node[right]  {$a=a_2$};

\clip (-.7,-4) rectangle (8,4);
\foreach \j in {1.75,2, 2.25, 2.5,3, 3.25,3.5,3.75, 4}  \draw (.5,\j*\j-2*\j-2) .. controls (4,\j*\j-2.1*\j-2.5).. (6.5,.8*\j*\j-2*\j-2.5);

\fill (1,-2.53) circle (.075);
\fill (1,-2.1) circle (.075);
\fill (1,-1.54) circle (.075);
\fill (1,-.85) circle (.075);

\fill  (6,-3.45) circle (.075);
\fill  (6,-3.18) circle (.075);
\fill (6,-2.78) circle (.075);
\fill (6,-2.3) circle (.075);
\fill[red] (6,-.99) circle (.075);
\fill[red] (6,-.19) circle (.075);

\fill[red] (4.15,0) circle (.075);
\fill[red] (5.74,0) circle (.075);

\end{tikzpicture}
\caption{%
Diagram showing the variation of eigenvalues $\mu(a, h)$ of $P_{a,h}$ as a function
of $a$ for fixed $h$. The eigenvalues are strictly decreasing in $a$. Consequently, the number of negative eigenvalues of $P_{a_2, h}$ is equal to the number of negative eigenvalues of $P_{a_1, h}$ together with the number that cross the $a$-axis between $a=a_1$ and $a=a_2$.}
\label{evals}\end{figure}
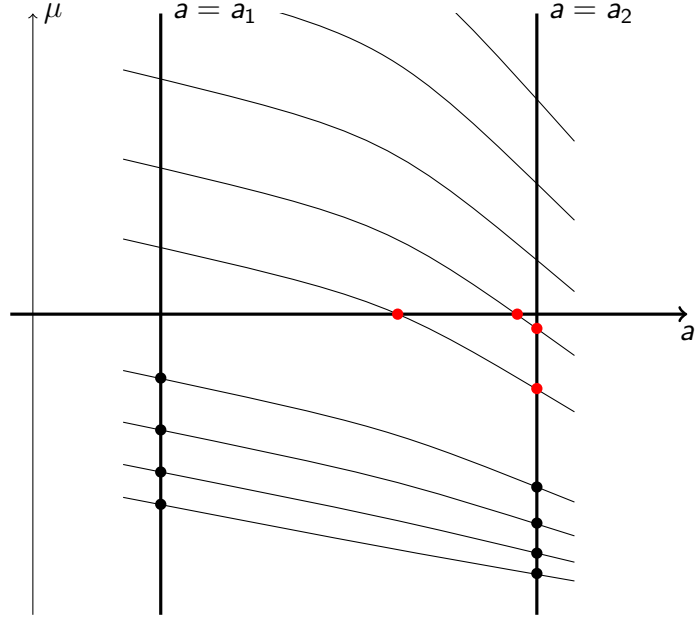

The strict monotonicity of $\mu(a, h)$ in $a$ shows that the number of eigenvalues $\mu_n(a,h)$ of $P_{a, h}$ that change from nonnegative to negative as $a$ varies from $a_1$ to $a_2$ is the same as the number of 
$\mu(a, h)$ (counted with multiplicity) equal to zero, for $a \in [a_1, a_2)$. 
Next, we observe that the space of  eigenfunctions $u_n(a, h)$ of $P_{a, h}$ with zero eigenvalue, i.e. $\mu_n(a, h) = 0$ is in one-to-one correspondence with the space of eigenfunctions of $C(\lambda)$, $\lambda = h^{-1}$, with eigenvalue $(a-i)(a+i)^{-1}$, or equivalently $e^{i\theta}$ where $a=-\cot(\theta/2)$. Indeed, whenever $u_n$ is such an eigenfunction of $P_{a, h}$, then 
\begin{equation}
f \Def \frac1{2} (h \partial_\nu u - iu) \big|_{\partial M}
\label{fu}\end{equation}
is an eigenfunction of $C(\lambda)$, with eigenvalue $(a-i)(a+i)^{-1}$. Conversely, if $f$ is an eigenfunction of $C(\lambda)$ with eigenvalue $(a-i)(a+i)^{-1}$, then by definition there exists a Helmholtz function $u$ such that $u$ is related to $f$ according to \eqref{fu}, and we have $(h \partial_\nu + a)u = 0$ at $\partial M$. 
This completes the proof.
\end{proof}

\begin{remark}\label{rem-3-1}
We can apply similar arguments for $\lambda^{-\delta}\rho^{-1} R(\lambda)$ as $\rho>0$ is a smooth function on $\partial M$ and then plug corresponding
parameters in the boundary conditions coming again to equality (\ref{eq-1-16}).
\end{remark}

We next digress
 to prove that the eigenvalues of $C(\lambda)$ are monotonic (that is, they move monotonically around the unit circle) in $\lambda$. This plays no role in the remainder of our proof, but is (in the authors' opinion) of independent interest. 

\begin{proposition}\label{prop:monotone}
The eigenvalues of $C(\lambda)$ rotate clockwise around the unit circle as $\lambda$ increases. 
\end{proposition}

\begin{remark} This implies that the eigenvalues of $R_{\scl}(\lambda)$ are monotone decreasing in $\lambda$. 
\end{remark}

\begin{proof}
As discussed in the previous proof, $C(\lambda)$ has eigenvalue $e^{i\theta}$ if and only if $P_{a, h}$ has a zero eigenvalue, where $a = a(\theta) = - \cot (\theta/2)$. Thus, as a function of $h = \lambda^{-1}$, $\theta(h)$ is defined implicitly by the condition
$$
\mu(a(\theta), h) = 0.
$$
Since $a$ is a strictly increasing function of $\theta$, and we have just seen that $\mu$ is a strictly decreasing function of $a$, it suffices to show that when $\mu = 0$, $\mu$ is a strictly increasing function of $h$, hence a strictly decreasing function of $\lambda$. 

We now compute the derivative of $\mu$ with respect to $h$, at a value of $a$ and $h$ where $\mu(a, h) = 0$. We have 
\begin{multline*}
\frac{d}{dh} \mu(a,h) =  \frac{d}{d h} ( (h^2 \Delta - 1) u(h), u(h) )_M \\[3pt]
= 2h ( \Delta u, u )_M  + ( (h^2 \Delta - 1) u' (h), u(h) )_M + ( (h^2 \Delta - 1) u(h), u' (h) )_M\\[3pt]
= 2h ( \Delta u, u)_M  + ( (h^2 \Delta - 1) u', u )_M - ( u' ,(h^2 \Delta - 1) u)_M.
\end{multline*}
In the third line, we used the fact that $(h^2 \Delta - 1) u = 0$ when $\mu(h) = 0$.
Note the second term is not zero, as $u'$ is not in the domain of the operator due to the changing boundary condition, so we cannot move the operator to the right hand side of the inner product without incurring boundary terms.
We use the Gauss-Green formula to express the last two terms as a boundary integral:
\begin{gather*}
\mu'(h) = 2h ( \Delta u, u )_M    +h(h\partial_\nu  u', u)_{\partial M}  - h(u',  h\partial_\nu u)_{\partial M}\\
\phantom{\mu'(h} = 2h ( \Delta u, u )_M    +h(h\partial_\nu  u', u)_{\partial M}  + ha(u',   u)_{\partial M}.\\
\intertext{Differentiating the boundary condition we find that }
(h\partial_\nu   + a ) u' = -h\partial_\nu  u \text{ at } \partial M.\\
\intertext{Substituting that in we get }
\mu'(h)
= 2h ( \Delta u, u )_M - h(u, \partial_\nu u)_{\partial M}. \\
\intertext{Applying Gauss-Green again, we get }
\mu'(h)
= h ( \Delta u, u )_M + h \| \nabla u\|_M^2  \\
\phantom{\mu'(h) +1+11} = h^{-1} \Big( \| u \|_{L^2(M)}^2 +  \| h\nabla u\|_{L^2(M)}^2 \Big)  > 0.
\end{gather*}
\end{proof}


\chapter{Semiclassical spectral asymptotics}
\label{sect-4}
In this section, we prove Theorem~\ref{thm-1-2}. Essentially, we have arrived at a rather  standard semiclassical spectral asymptotics problem and results are due to \cite{Ivr1}, Chapter 5 or \cite{futurebook}, Chapter \ref{book_new-sect-7}.

\begin{proposition}\label{prop-4-1}
\begin{enumerate}[fullwidth, label=(\roman*)]
\item\label{prop-3-1-i}
Let $\N^-_h (a)$ be as in \eqref{eq-1-17}.
The following asymptotic holds as $h\to +0$:
\begin{equation}
\N^-_h (a) = (2\pi h)^{-d} \omega_d \vol (M) + O(h^{1-d})
\label{eq-4-1}
\end{equation}
\item\label{prop-3-1-ii}
Further, if the set of periodic billiards on $M$ has measure $0$ then as $h\to +0$:
\begin{gather}
\N^-_h (a) = (2\pi h)^{-d} \omega_d \vol (M) +  h^{1-d}  \kappa(a) \vol'(\partial M)   + o(h^{1-d})
\label{eq-4-2}
\end{gather}
with $\kappa(a)$ given by \eqref{eq-1-20}.
\end{enumerate}
\end{proposition}

\begin{proof}
One can check easily that the operator $P_{a,h}$ is microhyperbolic at energy level $0$ at each point $(x,\xi)\in T^*M$ in the direction $\xi$; further, the boundary value problem is microhyperbolic at each point $(x'; \xi')\in T^*\partial M$ at energy level $0$ in the multidirection $(\xi', \xi_1^-,\xi_1^+)$ with $\xi_1=\xi_1^\pm $ roots of $\sum g^{jk}\xi_j\xi_k=0$; finally, the boundary value problem is elliptic at each point of the elliptic zone ($\subset T^*\partial M$) if $a\le 0$, and either elliptic or microhyperbolic  in the direction $\xi'$ at each point of the elliptic zone ($\subset T^*\partial M$) if $a> 0$ --- see definitions in Chapters \ref{book_new-sect-2},  \ref{book_new-sect-3} of \cite{futurebook}. Then statements \eqref{eq-1-18}, \eqref{eq-1-19} follow from Theorems \ref{book_new-thm-7-3-11} and \ref{book_new-thm-7-4-1} of \cite{futurebook}.

We now assume that the set of periodic billiards on $M$ has measure zero, and compute the second term in the spectral asymptotic explicitly. Similar calculations appear in \cite{FG}. 

To do this, one can use method of freezing coefficients (see f.e.  \cite{futurebook}, \ref{book_new-sect-7-2})  which results in
\begin{equation}
h^{1-d} \kappa (a)=\int_{\bR^+} \bigl(e (0, x_1; 0, x_1;1) - (2\pi h )^{-d}\omega_d \bigr)\, dx_1
\label{eq-4-3}
\end{equation}
where $e(x',  x_1; y', y_1;\tau )$ is the Schwartz kernel of the spectral projector $E(\tau)$ of the operator
$H_a=h^2 \Delta$ in half-space $\bR^{d-1}\times \bR^+\ni (x',x_1)$ with domain $\mathfrak{D}(H_a)= \{u \in H^2:\ (h\partial_{x_1} +a)u|_{x_1=0}=0\}$.

We obtain this spectral projector by integrating the spectral measure. This in turn is obtained via Stone's formula
\begin{equation}
dE_L(\sigma) = \frac{1}{2\pi i}\Big( (L - (\sigma + i0))^{-1} - (L - (\sigma - i0)^{-1} \Big) \, d\sigma.
\label{eq-4-4}\end{equation}

Consider the resolvent for $H_a$, $(H_a - \sigma)^{-1}$, for $\sigma \in \CC \setminus \RR$. Using the Fourier transform in the $x'$ variables, we can write the Schwartz kernel of this resolvent in the form
\begin{equation}
(2\pi h)^{1-d} \int e^{i(x'-y') \cdot \xi'} (T_a + |\xi'|^2 - \sigma)^{-1}(x_1, y_1) \, d\xi'.
\label{eq-4-5}
\end{equation}
Here $T_a$ is the one-dimensional operator $T_a=-h^2\partial ^2 +|\xi'|^2$ on $L^2(\RR_+)$ under the boundary condition $(h\partial +a)u|_{x_1=0}=0$.  The spectral projector $E_{H_a}(1)$ is therefore given by
\begin{equation}
(2\pi h)^{1-d} \int_{-\infty}^1 \int e^{i(x'-y') \cdot \xi'} dE_{T_a}(\sigma - |\xi'|^2)(x_1, y_1) \, d\xi' \, d\sigma.
\label{eq-4-6}\end{equation}

Thus, we need to find the spectral measure for $T_a$. Write $\sigma - |\xi'|^2 = \eta^2$, where we take $\eta$ to be in the first quadrant of $\CC$ for $\Im \sigma > 0$, and in the fourth quadrant for $\Im \sigma < 0$.

\begin{lemma}\label{lemma-4-2}
Suppose that $\Im \eta > 0$ and $\Re \eta \geq 0$. Then the resolvent kernel $(T_a - \eta^2)^{-1}$ takes the form
\begin{equation}
(T_a - \eta^2)(x, y) = \begin{cases}
\frac{i}{2h\eta} \Big( e^{i\eta(x-y)/h} + \frac{i\eta - a}{i\eta + a} e^{i\eta(x+y)/h} \Big) , \quad x > y \\[3pt]
\frac{i}{2h\eta} \Big( e^{i\eta(y-x)/h} + \frac{i\eta - a}{i\eta + a} e^{i\eta(x+y)/h} \Big), \quad x < y .
\end{cases}\label{eq-4-7}\end{equation}
If $\Im \eta < 0$ and $\Re \eta \geq 0$, then the resolvent kernel $(T_a - \eta^2)^{-1}$ takes the form
\begin{equation}
(T_a - \eta^2)(x, y) = \begin{cases}
-\frac{i}{2h\eta} \Big( e^{i\eta(y-x)/h} + \frac{i\eta + a}{i\eta - a} e^{-i\eta(x+y)/h} \Big), \quad x > y \\[3pt]
-\frac{i}{2h\eta} \Big( e^{i\eta(x-y)/h} + \frac{i\eta + a}{i\eta - a} e^{-i\eta(x+y)/h} \Big), \quad x < y .
\end{cases}\label{eq-4-8}\end{equation}
\end{lemma}

\begin{proof}
In the regions $x < y$ and $x > y$, the resolvent kernel must be a linear combination of $e^{i\eta x/h}$ and $e^{-i\eta x/h}$.
Moreover, for $\Im \eta > 0$, we can only have the $e^{+i\eta x/h}$ term, as $x \to \infty$, as the other would be exponentially increasing.
So we can write the kernel in the form
\begin{equation}
\left\{\begin{aligned}
& c_1 e^{+i\eta x/h}, &&x > y \\
&c_2 e^{+i\eta x/h} + c_3 e^{-i\eta x/h}, &&x < y.
\end{aligned}\right.
\label{eq-4-9}\end{equation}
We apply the boundary condition, and the two connection conditions at $x=y$, namely  continuity, and a jump in the derivative of $-1/h$, in order to obtain the delta function $\delta(x-y)$ after applying $T_a$. These three conditions determine the $c_i$ uniquely, and we find that, in the case $\Im \eta > 0$,
\begin{phantomequation}\label{eq-4-10}\end{phantomequation}
\begin{align}
&c_1 = \frac{i}{2h\eta} \big( e^{-i\eta y/h} + \frac{i\eta-a}{i\eta + a}
 e^{+i\eta y/h}\big),\label{eq-4-10-1}\tag*{$(\ref*{eq-4-10})_1$}\\
&c_2 =  \frac{i}{2h\eta} \frac{i\eta-a}{i\eta + a}  e^{+i\eta y/h}, \qquad c_3 = \frac{i}{2h\eta} e^{+i\eta y/h},
\label{eq-4-10-2}\tag*{$(\ref*{eq-4-10})_{2,3}$}
\end{align}
yielding \eqref{eq-4-7}. A similar calculation yields \eqref{eq-4-8}.
\end{proof}

We now apply \eqref{eq-4-4} to find the Schwartz kernel of the spectral measure for $T_a$.

\begin{lemma}\label{lemma-4-3}
The spectral measure $dE_{T_a}(\tau)$ is given by the following.
\begin{enumerate}[label=(\roman*), fullwidth]
\item \label{lemma-4-3-i}
For $\tau \geq 0$, $\tau = \eta^2$
\begin{multline}
dE_{T_a}(\tau)\\=\frac{1}{4\pi h \eta} \Big( e^{i\eta(x-y)/h} + e^{i\eta(y-x)/h} + \frac{i\eta - a}{i\eta + a} e^{i\eta(x+y)/h} + \frac{i\eta + a}{i\eta - a} e^{-i\eta(x+y)/h}  \Big) \, 2\eta d\eta.
\label{eq-4-11}\end{multline}
\item\label{lemma-4-3-ii}
For $\tau < 0$,  the spectral measure $dE(\tau)$ vanishes for $a \leq 0$, while for $a > 0$
\begin{equation}
dE_{T_a}(\tau)=\frac{2a}{h} e^{-ax/h}e^{-ay/h} \delta(\tau + a^2) d\tau.
\label{eq-4-12}\end{equation}
\end{enumerate}
\end{lemma}

\begin{proof} This follows directly from Lemma~\ref{lemma-4-2} and Stone's formula, \eqref{eq-4-4}. The extra term for $a > 0$ arises from the pole in the denominator, $i\eta + a$ for $\Im \eta > 0$ and $i\eta - a$ for $\Im \eta < 0$ in the expressions \eqref{eq-4-7}, \eqref{eq-4-8}, which only occurs for $a >0$. For $\tau$ negative, we need to set $\eta = i \sqrt{-\tau} + 0$ in \eqref{eq-4-7} and $\eta = - i\sqrt{-\tau} + 0$ in \eqref{eq-4-8}, and subtract. Then everything cancels except at the pole, where we obtain a delta function $-2\pi i\delta(\sqrt{-\tau} - a)$,  which arises from $(\sqrt{-\tau} + i0 + a)^{-1} - (\sqrt{-\tau} - i0 + a)^{-1}$. This term arises from the negative eigenvalue $-a^2$ which occurs for $a > 0$, corresponding to the eigenfunction $\sqrt{2a/h} \, e^{-ax/h}$.
\end{proof}

Plugging this into \eqref{eq-4-6}, and making use of the fact that $d\sigma d\xi'= 2\eta d\eta d\xi'$, we find that the Schwartz kernel of $E_{H_a}(1)$ is given by
\begin{multline}
(2\pi h)^{-d} \int_{0}^1 \int H(1 - |\xi'|^2 - \eta^2) e^{i(x'-y') \cdot \xi'} \\
\times \Big( e^{i\eta(x_1-y_1)/h} + e^{i\eta(y_1-x_1)/h} + \frac{i\eta - a}{i\eta + a} e^{i\eta(x_1+y_1)/h} + \frac{i\eta + a}{i\eta - a} e^{-i\eta(x_1+y_1)/h}  \Big)  \, d\xi' \,  d\eta
\label{eq-4-13}\end{multline}
for $a \leq 0$ while for $a > 0$, it is given by the sum of \eqref{eq-4-13} and
\begin{multline}
(2\pi h)^{1-d} \frac{2a}{h} e^{-ax_1/h} e^{-ay_1/h} \int_{-\infty}^1 \int e^{i(x' - y') \cdot \xi'} \delta(\sigma - |\xi'|^2 + a^2) \, d\xi' \, d\sigma .
\label{eq-4-14}
\end{multline}
We are actually interested in the value on the diagonal. Setting $x=y$, and performing the trivial $\xi'$ integral, we find that the Schwartz kernel of the spectral projector $E_{H_a}(1)(x,x)$ on the diagonal is given by
\begin{multline}
\frac{\omega_{d-1}}{ (2\pi h)^{d}} \int_{0}^1 (1 - \eta^2)^{(d-1)/2} \Bigg( 2 + \frac{i\eta - a}{i\eta + a} e^{2i\eta x_1/h}  + \frac{i\eta + a}{i\eta - a} e^{-2i\eta x_1/h}  \Bigg)   \,  d\eta  \\
+ H(a) \frac{(d-1) \omega_{d-1}}{ (2\pi h)^{d-1}} \frac{a}{h} e^{-2ax_1/h} \int_{-a^2}^1 (\sigma+a^2)^{(d-3)/2} \, d\sigma .
\label{eq-4-15}\end{multline}
Since
\begin{equation*}
\omega_{d-1}\int_0^1 2(1 - \eta^2)^{(d-1)/2} \, d\eta = \omega_d,
\end{equation*}
we see by comparing with \eqref{eq-4-3} that this term disappears in the expression for $\kappa(a)$ and we have, after performing the $x_1$ integral as in \eqref{eq-4-3}
\begin{multline}
h^{1-d} \kappa(a) = \frac{\omega_{d-1}}{ (2\pi h)^{d}} \int_{0}^1 (1 - \eta^2)^{(d-1)/2} \\
\shoveright{\times \Bigg( \frac{i\eta - a}{i\eta + a} \Big( \frac{ih}{2} (\eta + i0)^{-1} \Big) \ - \frac{i\eta + a}{i\eta - a}  \Big( \frac{ih}{2} (\eta - i0)^{-1} \Big) \Bigg)   \,  d\eta  }\\
+ H(a) \frac{(d-1) \omega_{d-1}}{2 (2\pi h)^{d-1}}  \int_{-a^2}^1 (\sigma+a^2)^{(d-3)/2} \, d\sigma .
\label{eq-4-16}\end{multline}
Simplifying a bit, and performing the $\sigma$ integral, we have
\begin{multline}
\kappa(a) = -\frac{i\omega_{d-1}}{2 (2\pi)^{d}} \int_{-1}^1 (1 - \eta^2)^{(d-1)/2}  \frac{(i\eta - a)^2}{a^2 + \eta^2} (\eta + i0)^{-1}   \,  d\eta  \\
+ H(a) \frac{\omega_{d-1}}{ (2\pi)^{d-1}}  (1+a^2)^{(d-1)/2}  .
\label{eq-4-17}\end{multline}
We further simplify this expression by expanding $(i\eta - a)^2 = a^2 - 2ia\eta - \eta^2$, and noting that the contribution of the $-\eta^2$ term is zero, as this gives an odd integrand in the $\eta$ integral. A similar statement can be made for the $a^2$ term, except that there is a contribution from the pole in this case. This leads to the expression
\begin{multline}
\kappa(a) = \frac{\omega_{d-1}}{(2\pi)^{d-1}} \bigg( -\frac{1}{2\pi} \int_{-1}^1 (1 - \eta^2)^{(d-1)/2}  \frac{a}{a^2 + \eta^2}  \,  d\eta  \\
- \frac{1}{4} + H(a)   (1+a^2)^{(d-1)/2} \bigg) .
\label{eq-4-18}\end{multline}
Although not immediately apparent, this formula is continuous at $a=0$. In fact, the function $a (a^2 + \eta^2)^{-1}$ has a distributional limit $(\operatorname{sgn} a) \pi \delta(\eta)$ as $a$ tends to zero from above or below.
The change of sign as $a$ crosses $0$ means that the integral in \eqref{eq-4-18} has a jump of $-1$ as $a$ crosses zero from negative to positive. That exactly compensates the jump in the final term.

In odd dimensions, we can compute this integral exactly. In particular, in dimension $d=3$, we find that
\begin{equation}
\kappa(a) = \frac{\omega_2}{(2\pi)^2} \Big( -\frac{1}{4} + \frac{a}{\pi} + (1+a^2) \big(1 - \frac{\arccot a}{\pi} \big) \Big).
\label{eq-4-19}
\end{equation}
\end{proof}

\begin{proof}[Proof of Theorem~\ref{thm-1-2}] This follows immediately from Proposition~\ref{prop-4-1} and Proposition~\ref{prop-1-1}.
\end{proof}

\begin{remark}\label{rem:FG} The second term of the expansion  in \cite[Theorem 1.1]{FG} is obtained by computing
\begin{equation}
(2\pi h)^{1-d} \int_{-\infty}^1 (1-\sigma) \int e^{i(x'-y') \cdot \xi'} dE_{T_a}(\sigma - |\xi'|^2)(x_1, y_1) \, d\xi' \, d\sigma
\end{equation}
instead of \eqref{eq-4-6}. 
\end{remark}

\chapter{Relation to Dirichlet boundary condition}
\label{sect-5}

In this section we observe that the limit $a \to -\infty$ corresponds to the Dirichlet boundary condition. More precisely, we have

\begin{proposition}\label{prop-5-1} Let $\N_h^-(-\infty)$ denote the limit
\begin{equation*}
\N_h^-(-\infty) \Def \lim_{a \to -\infty} \N_h^-(a),
\end{equation*}
where $\N_h^-(a)$ is given by \eqref{eq-1-17}. Then we have
\begin{equation}
\N_h^-(-\infty) =
\# \{ \lambda_j \leq h^{-1} \mid \lambda^2_j
\text{ is a Dirichlet eigenvalue of  } \Delta \}.
\label{eq-5-1}
\end{equation}
\end{proposition}

\begin{remark}\label{rem-5-2}
 Because the quadratic form \eqref{eq-1-15} is monotone in $a$, the counting function $\N_h^-(a)$ is monotone in $a$. Hence the limit above exists.
\end{remark}

\begin{proof} We use the min-max characterisation of eigenvalues.
Let $\tilde{\N}_D(\lambda)$ denote the number of Dirichlet eigenvalues (counted with multiplicity) less than or equal to $\lambda = h^{-1}$. This is equal to the maximal dimension of a subspace of $H^1_0(M)$ on which the quadratic form $Q_D$, given by
\begin{equation}
Q_D(u, u) = h^2 \| \nabla u \|_2^2 - \| u \|_2^2
\label{eq-5-2}\end{equation}
is negative semidefinite. On the other hand, $\tilde{\N}_h^-(a)$ is equal to the maximal dimension of a subspace of $H^1(M)$ on which the quadratic form $Q_a$ given by \eqref{eq-1-15} is (strictly) negative definite.

We first show that $\tilde{\N}_D(h^{-1}) \leq \tilde{\N}_h^-(-\infty)$.
Let $V$ be the vector space spanned by Dirichlet eigenfunctions with eigenvalue $\leq \lambda^2$.
Clearly, the quadratic form $Q_a$ is negative \emph{semi}definite on $V$, and if $\lambda^2$ is not a Dirichlet eigenvalue, then it is negative definite, proving the assertion. In the case that $\lambda^2$ is a Dirichlet eigenvalue, we perturb $V$ to $V_\epsilon$, a vector space of $H^1(M)$ of the same dimension as $V$, so that, for  for $\epsilon$ sufficiently small depending on $a$,  $Q_a$ is negative definite on $V_\epsilon$. For simplicity we only do this in the case that the $\lambda^2$-eigenspace is one dimensional, leaving the general case to the reader.
To do this, we choose an orthonormal basis of $V$ (with respect to the $L^2$ inner product) of Dirichlet eigenfunctions $v_1, \dots, v_k$ with eigenvalues $\lambda_1^2 \dots \lambda_k^2$, where $\lambda_k = \lambda$.
Then we perturb only $v_k$, leaving the others fixed.
We choose $s \in H^1_0(M) ^\perp$, the orthogonal complement of $H^1_0(M)$ in $H^1(M)$ (with respect to the inner product in $H^1(M)$), so that \begin{equation}
Q_a(v_i, s) = 0, \quad i < k \text{ and } Q_a(v_k, s) > 0.
\label{eq-5-3}\end{equation}
We check that this is possible. Notice that $s \in H^1_0(M) ^\perp$ implies that $(\Delta +1) s = 0$ in $M$. Then as $v_i$ has zero boundary data, we have
\begin{equation}
(\lambda_i^2 + 1) ( v_i, s)_M = ( \Delta v_i, s )_M - (  v_i, \Delta s )_M \\
= \langle \partial_\nu v_i, s \rangle_{\partial M} .
\label{eq-5-4}\end{equation}
We choose $s$ so that $\langle \partial_\nu v_i, s \rangle_{\partial M}$ vanishes for   $i < k$ and
is positive for $i=k$. This is  possible: in fact, due to the unique solvability of the boundary value problem
\begin{equation}
(\Delta + 1)s = 0, \quad s |_{\partial M} = f \in H^{1/2}(M),
\label{eq-5-5}\end{equation}
for $s \in H^1(M)$, we see that
$s$ can have any boundary value in $H^{1/2}(\partial M)$ which is dense in $L^2(\partial M)$. Then using \eqref{eq-5-4} we see that $\langle \partial_\nu v_k, s \rangle_{\partial M} > 0$ implies that $( v_k, s )_M >0$.

We now define $V_\epsilon$ to be the span of $v_1, \dots, v_{k-1}$ and $v_k + \epsilon s$.
Then we have
\begin{equation}
Q_a(v_i, v_k + \epsilon s) = 0, \ i < k
\label{eq-5-6}
\end{equation}
and
\begin{align}
Q_a(v_k+ \epsilon s,v_k + \epsilon s ) &=
Q_a(v_k, v_k) + 2\epsilon Q_a(v_k, s) + \epsilon^2 Q_a(s, s) \label{eq-5-7}\\
&= 2\epsilon Q_a(v_i, s_i) + \epsilon^2 Q_a(s_i, s_i) \notag\\
&= - 2 \epsilon (h^2 + 1) (v_i, s_i)_M + O(\epsilon^2 a^2)\notag
\end{align}
which is strictly negative for $\epsilon a^2$ small enough. It follows that $Q_a$ is negative definite on $V_k$ when $\epsilon a^2$ is small enough. A similar construction can be made when $\lambda^2$ has multiplicity greater than $1$.

We next show that $\tilde{\N}_D(h^{-1}) \geq \N_h^-(-\infty)$. We argue by contradiction: if not, then for any $a$, there is a vector space $W$ of dimension $\geq k+1$ on which $Q_a$ is negative definite. Then there is a nonzero vector $w \in W$ orthogonal (in the $H^1(M)$ inner product) to $V$. We can write $w = w' + s$ where $w' \in H^1_0(M)$ and $s \in H^1_0(M)^\perp$. Then $w'$ is a linear combination of Dirichlet eigenfunctions with eigenvalue $\geq \lambda'> \lambda$, where $\lambda'$ is the smallest eigenvalue larger than $\lambda$. We then have
\begin{multline}
0 > Q_a(w' + s, w' + s) = Q_a(w', w') + 2 Q_a(w', s) + Q_a(s, s) \\
\geq (\lambda' - \lambda) \| w' \|_2^2 - 2(h^2 + 1) \| w' \|_2 \| s \|_{L^2(M)} - h a \| s \|_{L^2(\partial M)}^2 .
\label{eq-5-8}
\end{multline}
However, some standard potential theory shows that $\| s \|_{L^2(M)}$ is bounded by a constant times $\| s\|_{L^2(\partial M)}$. To see this, extend $M$ to a larger manifold $\tilde M$ of the same dimension, and let $G(x,y)$ be the Schwartz kernel of the inverse of $(\Delta_{\tilde M} + 1)^{-1}$ on $L^2(\tilde M)$, with Dirichlet boundary conditions at $\partial \tilde M$. We can write $s$ as $\int_{\partial M} d_{\nu_y} G(x, y) h(y) \, dy$ where $(1/2 + D)h = s |_{\partial M}$ and $D$ is the double layer operator on $\partial M$ determined by $G$. Standard arguments show that $(1/2 + D)$ has a bounded inverse on $L^2(\partial M)$ and $d_{\nu_y} G(x, y)$ is a bounded integral operator from $L^2(\partial M)$ to $L^2(M)$.
So we can write, for $a < 0$,
\begin{multline*}
0 > Q_a(w' + s, w' + s)  \\
\geq (\lambda' - \lambda) \| w' \|_2^2 - 2C(h^2 + 1) \| w' \|_2 \| s \|_{L^2( \partial M)}  + h |a| \| s \|_{L^2(\partial M)^2}^2
\end{multline*}
and the RHS is clearly positive for $|a|$ large enough, giving us the desired contradiction.

\end{proof}


\providecommand{\bysame}{\leavevmode\hbox to3em{\hrulefill}\thinspace}

\vglue .06truein

\begin{tabular}{lll}
Andrew Hassell, &{\hskip 70 pt}  &Victor Ivrii \\
Mathematical Sciences Institute,&& Department of Mathematics\\
Australian National University,&&University of Toronto,\\
&&40, St.George Str.,\cr
Canberra, ACT 0200&&Toronto, Ontario M5S 2E4\cr
Australia&&Canada\cr
Andrew.Hassell@anu.edu.au&&ivrii@math.toronto.edu\cr
\end{tabular}


\begin{thebibliography}{BrIvr}



\bibitem[BH]{BH}
\textsc{A. Barnett and A. Hassell}, Boundary quasi-orthogonality and sharp inclusion bounds for large Dirichlet
eigenvalues. SIAM J. Numer. Anal. 49 (2011), no. 3, 1046--1063.

\bibitem[BFK]{BFK} 
\textsc{D.  Burghelea, L. Friedlander, T. Kappeler},  Meyer-Vietoris type formula for determinants of elliptic differential operators. J. Funct. Anal. 107 (1992), no. 1, 34--65.

\bibitem[Cal]{Cal}
\textsc{A. P. Calder\'on}, 
On an inverse boundary value problem. Seminar on Numerical Analysis and its Applications to Continuum Physics (Rio de Janeiro, 1980), pp. 65--73, Soc. Brasil. Mat., Rio de Janeiro, 1980. 


\bibitem[DGHH]{DGHH}
\textsc{K. Datchev, J.  Gell-Redman, A. Hassell, P. Humphries,},  Approximation and equidistribution of phase shifts: spherical symmetry. Comm. Math. Phys. 326 (2014), no. 1, 209--236.

\bibitem[DuGu]{DuGu} 
\textsc{J. J. Duistermaat, V. W.  Guillemin}, The spectrum of positive elliptic operators and periodic bicharacteristics. Invent. Math. 29 (1975), no. 1, 39--79. 

\bibitem[FG]{FG}
\textsc{R. L. Frank, and L.  Geisinger}, Semi-classical analysis of the Laplace operator with Robin boundary conditions. Bull. Math. Sci. 2 (2012), no. 2, 281--319. 

\bibitem[Fried]{Fried}
\textsc{L. Friedlander}, Some inequalities between Dirichlet and Neumann eigenvalues. Arch. Rational Mech. Anal. 116 (1991), no. 2, 153--160. 


\bibitem[GHZ]{GHZ}
\textsc{J. Gell-Redman, A. Hassell, S. Zelditch},  Equidistribution of phase shifts in semiclassical potential scattering. J. Lond. Math. Soc. (2) 91 (2015), no. 1, 159--179.

\bibitem[Ivr80]{Ivr80} 
\textsc{V.~Ivrii}, The second term of the spectral asymptotics for a Laplace-Beltrami operator on manifolds with boundary. (Russian) Funktsional. Anal. i Prilozhen. 14 (1980), no. 2, 25--34. English translation in Functional Analysis and Its Applications, 14:2 (1980), 98--106. 


\bibitem[Ivr1]{Ivr1}
\textsc{V.~Ivrii}.
\emph{Microlocal Analysis and Precise Spectral Asymptotics},
Springer-Verlag, SMM, 1998, xv+731.


\bibitem[Ivr2]{futurebook}
\textsc{V.~Ivrii}.
\emph{Microlocal Analysis and Sharp Spectral Asymptotics},
in progress: available online at \newline
\href{http://www.math.toronto.edu/ivrii/futurebook.pdf}{http://www.math.toronto.edu/ivrii/futurebook.pdf}

\bibitem[SoYa]{SoYa}
\textsc{A.~V. Sobolev and D.~R. Yafaev}, 
Phase analysis in the problem of scattering by a radial potential.
Zap. Nauchn. Sem. Leningrad. Otdel. Mat. Inst. Steklov. (LOMI),
147:155--178, 206, 1985.





%
\end{thebibliography}
\end{document}